\documentclass[11pt]{article}
\usepackage{amsfonts,amsmath,latexsym,color,epsfig}

\newtheorem{theorem}{Theorem}

\newtheorem{proposition}[theorem]{Proposition}

\newtheorem{corollary}[theorem]{Corollary}

\newenvironment{proof}
     {\medskip\noindent{\bf Proof:}\hspace{1mm}}
      {\hfill$\Box$\medskip}

\def\eps{\varepsilon}
\def\qed{\ifvmode\mbox{ }\else\unskip\fi\hskip 1em plus 10fill$\Box$}
\long\def\ignore#1{}

\begin{document}
\thispagestyle{empty}
\begin{titlepage}
\thispagestyle{empty}
\title{On the Richter-Thomassen Conjecture about\\ Pairwise Intersecting Closed Curves}

\author{J\'anos Pach\thanks{EPFL, Lausanne and R\'enyi Institute,
    Budapest. Supported by OTKA grant NN-102029 under EuroGIGA
    projects GraDR and ComPoSe, and by Swiss National
    Science Foundation Grants 200020-144531 and 20021-137574. Email: {\tt pach@cims.nyu.edu}}
  \and Natan Rubin\thanks{Department of Computer Science, Ben Gurion University of the Negev, Be\'{}er Sheva 84105, Israel. 
Email: {\tt rubinnat.ac@gmail.com}. Work on this paper was partly performed at Universit\'{e} Pierre \& Marie Curie and Universit\'{e} Paris Diderot, Institut de Math\'{e}matiques de Jussieu (UMR 7586 du CNRS), 4 Place Jussieu, 75252 Paris Cedex, France. N.R. was partly supported by the Fondation Sciences Math\'{e}matiques de Paris (FSMP) and by a public grant overseen by the French National Research Agency (ANR) as part of the "Investissements d\'{}Avenir" program (reference: ANR-10-LABX-0098).}
  \and G\'abor Tardos\thanks{R\'enyi Institute, Budapest. Supported by EPFL
    Lausanne, the Hungarian
    OTKA grant NN-102029 and an NSERC Discovery grant. Email: {\tt
      tardos@renyi.hu}}}
\thispagestyle{empty}
\maketitle
\begin{abstract}
\thispagestyle{empty}
A long standing conjecture of Richter and Thomassen
states that the total number of intersection points between any $n$ simple closed Jordan curves in the plane, so that any pair of them intersect and no three curves pass through the same point, is at least $(1-o(1))n^2$.

We confirm the above conjecture in several important cases, including the case (1) when all curves are convex, and (2) when the family of curves can be partitioned into two equal classes such that each curve from the first class is touching every curve from the second class. (Two curves are said to be touching if they have precisely one point in common, at which they do not properly cross.)

An important ingredient of our proofs is the following statement: Let $S$ be a family of the graphs of $n$ continuous real functions defined on $\mathbb{R}$, no three of which pass through the same point. If there are $nt$ pairs of touching curves in $S$, then the number of crossing points is $\Omega(nt\sqrt{\log t/\log\log t})$.
\end{abstract}
\thispagestyle{empty}
\maketitle
\end{titlepage}

\section{Introduction}
Studying combinatorial properties of arrangements of curves in the plane and surfaces in higher dimensions is one of the central themes of discrete and computational geometry, and has numerous applications in motion planning, ray shooting, computer graphics, pattern recognition, combinatorial optimization etc.~\cite{Ed87}, \cite{PaS09}, \cite{ShA95}. The analysis of the complexity of many geometric algorithms crucially depends on extremal results on arrangements. In the last forty years, powerful probabilistic, algebraic and algorithmic techniques have been developed to deal with such questions. The classical open problem of Erd\H os~\cite{Er46} on the maximum number of times the unit distance can occur among $n$ points in the plane can also be paraphrased as a problem on arrangements, in two different ways:
\begin{enumerate}
\item What is the maximum number of incidences between $n$ points and $n$ unit circles in the plane, where a point $p$ and a circle $c$ are called {\em incident} if $p$ belongs to $c$?
\item What is the maximum number of {\em touchings} among $n$ circles of radius $1/2$ in the plane? Two such circles are tangent to each other if and only if their centers are at unit distance.
\end{enumerate}

A seminal result in this area is the Szemer\'edi-Trotter
theorem~\cite{SzT83a}, \cite{SzT83b}, which states the number of incidences
between $n$ points and $m$ lines in the plane is $O(n^{2/3}m^{2/3}+m+n)$; for
many important generalizations and applications, see~\cite{BMP05},
\cite{SoT12}, \cite{TaV10}. Nearly 20 years ago, Sz\'ekely~\cite{Sz97} found
an elegant argument using crossing numbers of graphs, which shows that the
same result holds for {\em pseudo-segments}, that is, for systems of curves,
every pair of which intersect at most once. This observation opened the door
to generalizations to systems of more general curves: circles, {\em
  pseudo-circles}, that is, curves with two (or a bounded number of) pairwise
intersections. The first substantial step in this direction was taken by
Tamaki and Tokuyama~\cite{PseudoParabolas}. They defined a {\em lens} to be
the union of two arcs connecting the same pair of points along different
curves belonging to a family $S$, and proved that for {pseudo-circles} one can
cut the curves at a number of points proportional to the maximum number
$\nu(S)$ of non-overlapping lenses, so that the curves break into pieces, any
two of which intersect at most once. Then one can apply the
Szemer\'edi-Trotter bound on pseudo-segments to upper bound the number of
incidences between the curves and the points. An interesting instance of
pseudo-circles, whose study is motivated by robotic motion planing, is the
family of homothetic copies of a fixed convex set in the plane
\cite{KLPS86}. The line of research initiated by Tamaki and Tokuyama was
continued in~\cite{Chan1}, \cite{Chan2}, \cite{Chan3}, \cite{ArS02},
\cite{AgS05}, \cite{MaT06}. 
Note that any point of tangency (i.e., a {\it locally} non-transversal intersection) between the curves of $S$ can be viewed as a degenerate lense. 
Thus, the parameter $\nu(S)$ is at least as
large as the number of tangencies (``degenerate lenses'') between the curves in $S$, and in many cases it can bounded in terms of this quantity.

Other geometric problems that boil down to estimating the maximum number of tangencies between a set of curves in the plane are discussed in Agarwal {\it et al.}~\cite{PseudoCircles}. In particular, they showed that any family of $n$ pairwise intersecting pseudo-circles admits at most $O(n)$ tangencies. As a consequence of the spectacular success of techniques from algebraic geometry in this area, recently some of these problems have been revisited from an algebraic perspective~\cite{KaMS12}.

In this paper, we address an old conjecture of Richter and
Thomassen~\cite{RiT95}. To formulate this conjecture, we need to introduce
some terminology. Let $S$ be a collection of curves in the plane in {\em general position}, in the sense that no three curves pass through the same point and no two share infinitely many points.
Two curves $s_1,s_2\in S$ are said to {\em touch} each other if they have {\em precisely one} point in common, at which the curves do not properly cross.

\begin{figure}[htbp]
\begin{center}
\input{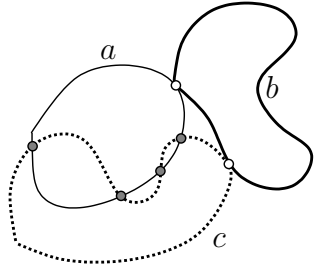}
\caption{\small In the depicted family of $3$ pairwise-intersecting curves, $b$ touches each of the remaining two curves $a$ and $c$, which cross one another.}
\label{Fig:LongDelaunay}
\end{center}
\end{figure}

The unique point that two touching curves have in common is called a {\em touching point} or, in short, a {\em touching}. All other points shared by two curves of the collection are called {\em crossing points} of $S$ or, simply, {\em crossings}, and the corresponding curves of $S$ are said to {\em cross} each other. The term {\em intersection} refers to both touchings and crossings. In what follows, we may assume that any point of tangency is necessarily a touching. Indeed, we can apply an arbitrary small perturbation to any pair $s_1,s_2\in S$ so as to remove all of their tangencies, with the possible exception of one last tangency, which then becomes a touching.


\medskip
\noindent{\bf Richter--Thomassen Conjecture} \cite{RiT95} {\it The total number of
  intersection points between $n$ pairwise intersecting simple closed curves
  in the plane which are in general position is at least $(1-o(1))n^2$}.
\medskip

\noindent Richter and Thomassen established the weaker lower bound $(3/4-o(1))n^2$, which was later improved
by Mubayi~\cite{Mu02} to $(4/5-o(1))n^2$. Observe that if every pair of curves
intersect at least twice, then the number of intersection points is at least
$2{n\choose2}=(1-o(1))n^2$. The Richter--Thomassen conjecture states that we
cannot substantially decrease the number of intersection points by allowing
touching pairs. Most likely, the presence of many such pairs significantly
increases the total number of intersection points.

The conjecture was confirmed by Salazar \cite{Sa99} in the case when every pair of curves have at most a bounded number of points in common. However, the problem has remained open for general families of simple closed curves, with perhaps the most intriguing special case involving $n$ {\it convex curves}.

\smallskip

The aim of this paper is to prove the Richter--Thomassen conjecture in this special case. More generally, we settle the question for collections of closed curves that can be decomposed into a constant number of $x$-monotone pieces. An open curve is called {\em $x$-monotone} if no two points of it have the same $x$-coordinate; see Figure \ref{Fig:Decomposable} (left).

\begin{figure}[htbp]
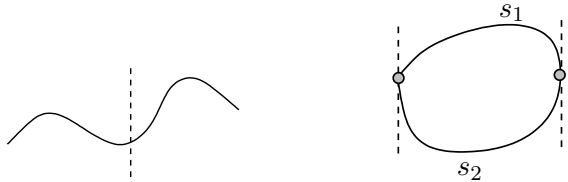

\begin{center}
\input{Monotone.pstex_t}\hspace{2cm}\input{Decomposable1.pstex_t}
\caption{\small Left: A curve is $x$-monotone if no two points have the same
  $x$-coordinate, so any vertical line can intersect it at most once. Right:
  Any convex curve can be decomposed into a pair of $x$-monotone
  curves $s_1$ and $s_2$.}
\label{Fig:Decomposable}
\end{center}
\end{figure}


\begin{theorem}\label{RichterThomassen}
Let $k$ be a fixed positive integer. The total number of intersection points between $n$ pairwise intersecting closed curves in general position in the plane, each of which can be decomposed into at most $k$ $x$-monotone curves, is at least $(1-o_k(1))n^2$.
\end{theorem}

Given a finite collection of closed convex curves in the plane, by slightly changing the direction of the $x$-axis, if necessary, we can assume that each element can be split into two $x$-monotone pieces; see Figure \ref{Fig:Decomposable} (right). Applying Theorem~\ref{RichterThomassen} with $k=2$, we establish the conjecture for all families of closed convex curves (that is, boundaries of convex regions).

\begin{corollary}\label{convex}
The total number of intersection points between $n$ pairwise intersecting closed convex curves in general position in the plane is at least $(1-o(1))n^2$.
\end{corollary}

Theorem~\ref{RichterThomassen} can be deduced from a general lower bound on the number of crossings between $n$ $x$-monotone curves in terms of the number of touchings between them. An $x$-monotone curve is called {\em bi-infinite} or {\em two-way infinite} if it is the graph of a continuous real function defined over the entire real line.

\begin{theorem}\label{bibi} Let $S_1$ and $S_2$ be two collections of bi-infinite $x$-monotone curves in general position. Suppose that $|S_1|=|S_2|=n$ and that the number of touching pairs $(s_1,s_2)$ with $s_1\in S_1$, $s_2\in S_2$ is $tn$, where $t$ is larger than an absolute constant $C$.

Then the number of crossing points between the elements of $S_1$ or the number of crossing points between the elements of $S_2$ is at least $\Omega\left(nt\sqrt{\log t/\log\log t}\right)$.
\end{theorem}

Let $S$ be a collection of $n$ bi-infinite $x$-monotone curves with $tn$ touching pairs. Partition $S$ randomly into two sets, $S_1$ and $S_2$, of size $\lfloor n/2\rfloor$ and $\lceil n/2\rceil$, respectively. The expected number of touchings between the curves in $S_1$ and the curves in $S_2$ will be more than $tn/2$. Fix a partition $S=S_1\cup S_2$ for which this number exceeds $tn/2$, and apply Theorem~\ref{bibi} to $S_1$ and $S_2$. We obtain the following.

\begin{corollary}\label{biinf} Let $S$ be a collection of $n$ bi-infinite $x$-monotone curves in general position with $tn$ touching pairs, where $t$ is bigger than an absolute constant. Then the number of crossings between the elements of $S$ is $\Omega\left(nt\sqrt{\log t/\log\log t}\right)$.
\end{corollary}

At the beginning of Section~\ref{second}, we will show that Corollary~\ref{biinf} implies the following statement for not necessarily bi-infinite curves.

\begin{corollary}\label{xmon} Let $S$ be a collection of $n$ $x$-monotone curves in general position, with at least $\epsilon n^2$ touching pairs. If $\epsilon>C\sqrt{\log\log n/\log n}$ for a suitable absolute constant $C$, then $S$ determines $\Omega\left(\epsilon n^2\sqrt{\log n/\log\log n}\right)$ crossing points.
\end{corollary}

Corollary \ref{xmon} easily extends to the following more general result, which immediately implies Theorem \ref{RichterThomassen}.

\begin{theorem}\label{xmonk} Let $S$ be a collection of $n$ closed curves in general position and with at least $\epsilon n^2$ touching pairs, so that each curve in $S$ can be decomposed into $k$ $x$-monotone curves. Suppose that $\epsilon>C_k\sqrt{\log\log n/\log n}$ for a suitable constant $C_k>0$ which may depend on $k$. Then $S$ determines $\Omega_k\left(\epsilon n^2\sqrt{\log n/\log\log n}\right)$ crossing points.
\end{theorem}

It is instructive to compare Theorems \ref{RichterThomassen} and \ref{xmonk}
in the case $S$ consists of pairwise-intersecting curves.
Note that the statement of Theorem \ref{RichterThomassen} can still hold even if both the number of crossings and the number of touchings determined by $S$ are quadratic in $n$ (for arbitrary large $n$). In contrast,
Theorem \ref{xmonk} guarantees that, in the above setup, the overall number of
crossing points must asymptotically {\it exceed} the number of touching
pairs. In particular, given a sufficiently large number $n$ of curves, each
decomposable into $k$ $x$-monotone curves, and
$\Omega(n^2)$ touchings among them, then some pair of curves must cross {\it
super-constantly} many times. For this conclusion, we can trade the assumption
of subdivisibility to a bounded number of $x$-monotone curves to a different
assumption: that each pair of curves intersect, see remark after Theorem~\ref{bipartite}.

\medskip
We prove Theorem~\ref{bibi}, which is the main technical result of this paper,
in Section \ref{third}. The proof is based on a double counting argument. We define a ``charging scheme'', that is, a weighted graph $G$ whose vertices correspond to the touchings and crossings between the curves in $S_1\cup S_2$. Theorem~\ref{bibi} follows by comparing our upper and lower bounds on the total weight of the edges.

\smallskip

Consider the setup of Theorem~\ref{bibi} in the special case when for every
pair $s_1\in S_1,s_2\in S_2$ the curve $s_1$ lies entirely below $s_2$ (with a
possible tangency). Fox {\em et al.}~\cite{FFPP10} established a tight bound
for this case. Namely, they proved that the number of crossing points between
the elements of $S_1$ or the number of crossing points between the elements of
$S_2$ is at least $\Omega(tn\log(tn))$, where, as above, $tn$ is the overall
number of tangent pairs $s_1\in S_1$ and $s_2\in S_2$. In particular, if every
element of $S_1$ touches all elements of $S_2$, we have $t=n$, so that
this bound becomes $\Omega(n^2\log n)$. In Theorem~\ref{bibi}, we prove a slightly weaker bound, but we drop the condition on the bipartite structure of the family of curves.

In Section~\ref{fifth}, we insist on a complete bipartite pattern
of tangencies, but we drop the condition of $x$-monotonicity.  We prove the following.

\begin{theorem}\label{bipartite}
Let $S_1$ and $S_2$ be collections of closed curves in general position. If $|S_1|=|S_2|=n$, any two elements of $S_1\cup S_2$ intersect and every element of $S_1$ touches all elements of $S_2$,
then the number of crossing points between the elements of $S_1$ or the number of crossing points between the elements of $S_2$ is $\Omega\left(n^2\sqrt{\log n/\log\log n}\right)$.
\end{theorem}

Note that this theorem also implies the statement formulated in the paragraph
after
Theorem~\ref{xmonk}. Namely, if $S$ is a collection of $n$ closed curves in
general position such that all pairs intersect and $\epsilon n^2$ pairs touch,
then some pair intersects super-constantly many times. Indeed, by the K\H
ov\'ari--S\'os--Tur\'an theorem, we find disjoint subsets $S_1$ and $S_2$ of $S$
with $|S_1|=|S_2|=\Omega(\log n)$ such that every curve in $S_1$ touches all
curves in $S_2$. Then, by Theorem~\ref{bipartite}, the curves in $S_1$ or $S_2$
have many crossings and thus there are two curves that intersect
$\Omega(\sqrt{\log\log n/\log\log\log n})$ times.

\section{Proofs of Corollary~\ref{xmon}, Theorems \ref{RichterThomassen} and \ref{xmonk} using Theorem~\ref{bibi}}\label{second}

We have seen in the Introduction that Corollary~\ref{biinf} is a direct
consequence of Theorem~\ref{bibi}. In this section, we first deduce
Corollary~\ref{xmon} from Corollary~\ref{biinf}. Then we show how
Theorem~\ref{RichterThomassen} and \ref{xmonk} follows from
Corollary~\ref{xmon}.

\medskip

\noindent{\bf Proof of Corollary~\ref{xmon}} (using Corollary~\ref{biinf}). Partition $S$ into two parts, $S_1$ and $S_2$, such that at least $\epsilon n^2/4$ touchings are {\em good} in the sense that they are formed by a curve from $S_1$ touching a curve from $S_2$ from below. This is possible, as the expected number of good touchings for a uniform random partition of $S$ into two parts is $\epsilon n^2/4$.

Perturb the curves in $S$ a little bit so that (1) they become $x$-monotone
polygonal paths consisting of finitely many straight-line segments, (2) all
crossing points and good touching points remain unchanged, and (3) all other
touchings are eliminated. Extend the resulting curves to bi-infinite polygonal
paths, as follows (see Figure \ref{Fig:Steep}). Let $z$ be a sufficiently
large positive number. Extend each curve in $S_1$ beyond its left endpoint and
each curve in $S_2$ beyond its right endpoint by a ray of slope $z$. Extend
each curve in $S_1$ beyond its right endpoint and each curve in $S_2$ beyond
its left endpoint by a ray of slope $-z$. By choosing $z$ large enough, we can
make sure that this last transformation preserves all good touching pairs and
creates only at most two new crossings for any pair of curves.\footnote{For
  instance, choosing $z$ larger than the absolute slopes of all straight-line
  segments that constitute the curves of $S$ will do.}
This means that we kept at least $\epsilon n^2/4$ touchings and created fewer than $n^2$ new crossing points.


\begin{figure}[htbp]
\begin{center}
\input{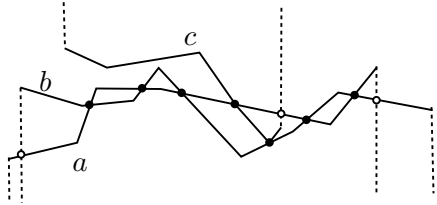}
\caption{\small Left: We extend each (polygonal) curve of $S_1\cup S_2$ to a bi-infinite polygonal path by sufficiently steep terminal rays.
(Among the three depicted curves, $a$ and $b$ belong to $S_1$, whereas $c$ belongs to $S_2$.) The extension introduces at most two additional crossings for any pair of curves.}
\label{Fig:Steep}
\end{center}
\end{figure}

In view of the condition on $\epsilon$, we can apply Corollary~\ref{biinf} to the modified bi-infinite curves with $t=\epsilon n/4$ to conclude that the number of crossing points is at least
$\Omega\left(nt\sqrt{\frac{\log t}{\log\log t}}\right)=\Omega\left(\frac{\epsilon n^2}{4}\sqrt{\frac{\log(\epsilon n/4)}{\log\log(\epsilon n/4)}}\right)=\Omega\left(\epsilon n^2\sqrt{\frac{\log n}{\log\log n}}\right).$
The lower bound on $\epsilon$ guarantees that over half of these points were already present in the original collection. \qed

\medskip
\noindent
{\bf Proof of Theorem~\ref{xmonk}} (using Corollary~\ref{xmon}).
Let $S$ be a collection of $n$ closed curves with at least $\eps n^2$ touching pairs as in Theorem \ref{xmonk}, so that each curve in $S$ can be decomposed into $k$ $x$-monotone curves. Clearly, dividing
the closed curves of $S$ into $x$-monotone pieces can introduce altogether at most
$kn$ cutpoints. The resulting collection of curves, $S'$, has still at least
$|T|-kn$ touching points, after losing all touchings that coincide with locally $x$-extremal
cutpoints of the curves. Assuming $\epsilon>C_k\sqrt{\log\log n/\log n}$ for a suitable constant $C_k$ guarantees that the new collection $S'$ satisfies the hypothesis of Corollary \ref{xmon} with $|S'|=n'=O_k(n)$ and $\eps'=\Omega_k(\eps)$.\qed

\medskip
\noindent
{\bf Proof of Theorem~\ref{RichterThomassen}} (using Theorem~\ref{xmonk}).
Two intersecting closed curves have at least two points in common, unless they
touch each other. Thus, the total number of intersection points between the
$n$ curves is at least $2{n\choose2}-|T|$, where $T$ denotes the set of
touching points. If $|T|=o(n^2)$, we are done. If this is not the case, we can
apply Theorem~\ref{xmonk} and conclude that the total number of crossings is
$\Omega(n^2\sqrt{\log n/\log\log n})$. \qed

\section{Proof of Theorem~\ref{bibi}}\label{third}

We say that a point $p$ is {\em to the right} of point $q$ and $q$ is {\em to the left} of $p$ in the plane if the $x$-coordinate of $p$ is larger than the $x$-coordinate of $q$.

Consider the collections $S_1$ and $S_2$ of bi-infinite $x$-monotone curves, as in the statement of Theorem~\ref{bibi}, and let $S=S_1\cup S_2$. By slightly perturbing the curves in $S$, if necessary, eliminate all touchings between them, except those at which a {\em curve in $S_1$ touches a curve in $S_2$ from above}. This operation does not change the set of crossing points, and we may assume without loss of generality that the number of remaining touchings is at least $tn/2$. Indeed, otherwise we can swap the roles of $S_1$ and $S_2$ and the assumption will hold. By a slight abuse of notation, we continue to denote the collections of perturbed curves $S_1$ and $S_2$, and their union by $S$. In the sequel, we will use the property that each curve in $S$ is touched only from one side.

\smallskip

Next, we construct a weighted bipartite multigraph $G$ such that its vertex classes are $T$, the set of {\em touching} points and $X$, the set of {\em crossing} points between pairs of curves coming from the same collection $S_1$ or $S_2$. That is, $X$ does not contain crossings between
a curve in $S_1$ and a curve in $S_2$.

\smallskip

Next, we define some sets of weighted edges, $A_k$, $B_k$ and $C_k$, running between $T$ and $X$. The edge set of $G$ is defined as the disjoint union of the sets $A_k$, $B_k$ and $C_k$, for all $1\leq k\leq t/2$ that are powers of $2$. Every edge between a touching $p\in T$ and a crossing $q\in X$ has the property that $p$ and $q$ belong to the same curve in $S$ and $p$ lies to the left of $q$.
In case $p$ and $q$ are
connected in several of these sets, we consider these distinct
edges and $G$ contains several parallel edges connecting $p$ and $q$,
typically with different weights.

The definitions of $A_k,B_k$ and $C_k$ are illustrated in Figures \ref{Fig:ChargeAB} and \ref{Fig:ChargeC}.
Namely, let $p\in T$ be the touching between two curves, $a,b\in S$, and let $q\in X$ be a crossing to the right of $p$, where $a$ crosses $c\in S$. Here we must have $a,c\in S_1$ and $b\in S_2$, or alternately $a,c\in S_2$ and $b\in S_1$. Let $\alpha>1$ be a parameter to be specified later.

\smallskip

\noindent {\bf A.} All edges of $A_k$
have weight $1/k$. The edge connecting a touching $p$ to a crossing $q$ is
present in $A_k$ if there are fewer than $k$ touching points $p'$ on $a$
between $p$ and $q$; see Figure \ref{Fig:ChargeAB} (left). 

\medskip
\noindent {\bf B.} All edges of $B_k$ have weight $\alpha/k$. The edge connecting $p$ and $q$ is present in $B_k$ if the following conditions are satisfied (refer to Figure \ref{Fig:ChargeAB}
(right)):

{\bf B1.} the curves $b$ and $c$ touch to the right of $q$ and

{\bf B2.} there are fewer than $k/\alpha$ touching points $p'$ on $a$ between
$p$ and $q$ with the property that the curve $b'\in S\setminus\{a\}$
containing $p'$ touches $c$.

\smallskip

\begin{figure}[htbp]
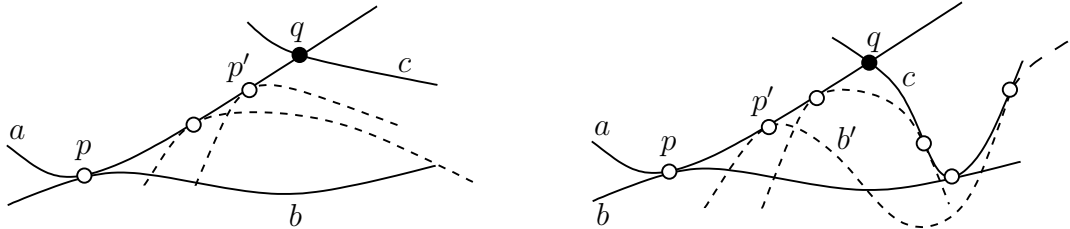

\begin{center}
\input{ChargeA.pstex_t}\hspace{1.5cm}\input{ChargeB.pstex_t}
\caption{\small The construction of $A_k$ and $B_k$ (resp., left and right). Each edge $(p,q)$ connects a touching $p\in T$ between $a$ and $b$ to a crossing $q\in X$ between $a$ and some third curve $c$.}
\label{Fig:ChargeAB}
\end{center}
\end{figure}

\begin{figure}[htbp]
\begin{center}
\input{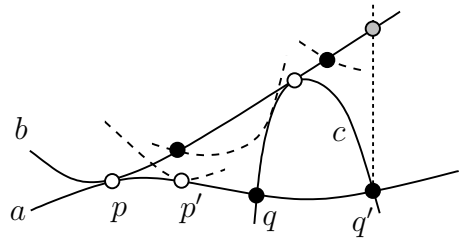}
\caption{\small The construction of $C_k$. The {\it arc} of $(p,q)\in C_k$ is the part of $b$ to the right of $p$ and left of $q'$.}
\label{Fig:ChargeC}
\end{center}
\end{figure}

\noindent {\bf C.} All edges of
$C_k$ have weight $\alpha^2/k$. The edge connecting $p$ and $q$ is present in $C_k$ if

{\bf C1.} the curves $a$ and $c$ cross to the right of $q$ with $q'$ being the next such crossing,

{\bf C2.} the curve $b$ touches $c$ between $q$ and $q'$,

{\bf C3.} there are fewer than $k$ touching points $p'$ on $a$ between $p$ and $q$, and

{\bf C4.} there are fewer than $\alpha k$ points of $X$ on $b$ to the right of
$p$ and to the left of $q'$.

We call the part of $b$ to the right of $p$ and left of $q'$ the {\em arc} of
the corresponding edge in $C_k$ (see Figure \ref{Fig:ChargeC}).

\smallskip

Our proof relies on a double counting of the total weight of the edges in $G$.
We prove an upper bound for the total weight of edges incident to crossing
points and a lower bound for the total weight of edges incident to most
touching points. This proof technique is commonly referred to as the
charging method.

\subsection{Upper bound on total weight}

We start with considering the edges incident to a fixed crossing point
$q\in X$ of two curves of $S$. Since $p$ can lie on either of these two curves, there are at most $2k$ edges $(p,q)\in A_k$ incident to $q$, for a
total weight of at most $2$. Summing over all possible values of $k$ we get a
total weight of at most $2\log t$.

Similarly, there are at most $2\lceil k/\alpha\rceil$ edges in $B_k$ incident
to $q$ for a total weight of at most $2+2\alpha/k$. Summing for all $k$ we get
that the total weight of these edges is less than $2\log t+4\alpha$.

For edges in $C_k$ incident to $q$, only one\footnote{This distinction depends on whether these two curves crossing at $q$ belong to $S_1$ or to $S_2$, and uses that the curves of $S_1$ are allowed to touch the curves of $S_2$ only {\it from above}.} of the two curves passing through
$q$ may play the role of $a$, the other curve must play the role of $c$.
Considering condition C3 in the definition of $C_k$ (and ignoring the arcs
of the edges for a moment), we see that $q$ is incident to at most $k$ edges of
$C_k$ for a total weight of at most $\alpha^2$.

\begin{proposition}\label{Prop:ArcsIntersect}
For any pair of edges $(p,q)\in C_k$ and $(p',q)\in C_{k'}$ incident to the same point $q\in X$,
their respective arcs either coincide or cross. In case of a crossing,
all the crossing points must belong to $X$.
\end{proposition}
\begin{proof}
Assume with no loss of generality that $q$ lies on $a\in S_2$, and let
$b,b'\in S_1$ be the respective second curves that are incident to $p$ and
$p'$. If $p=p'$, then the arcs of the two edges coincide. Otherwise we must have $b\neq b'$.
Let $p_1$ and $p_2$ be the respective points at which $b$ and $b'$ touch $c$ (as prescribed in condition C2).  The proposition now follows because the vertical lines through $p_1$ and $p_2$ intersect the curves $b$ and $b'$ in different order, each time within the respective arc of $(p,q)$ or $(p',q)$. See Figure \ref{Fig:ArcsIntersect}.
\end{proof}

\begin{figure}[htbp]
\begin{center}
\input{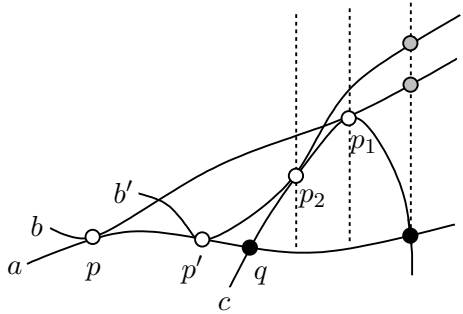}
\caption{\small Proof of Proposition \ref{Prop:ArcsIntersect}: The curves $b$ and $b'$ are tangent to $c$ at the respective points $p_1$ and $p_2$, so they must cross between the vertical lines through $p_1$ and $p_2$.}
\label{Fig:ArcsIntersect}
\end{center}
\end{figure}

Let $k_0$ be the smallest index such that an edge in $C_{k_0}$ is incident to
$q$. Consider the arc of any such edge. By condition C4, there are fewer than
$\alpha k_0$ crossing points on this arc that belong to $X$. The crossing
point with the arc of an edge in $C_k$ incident to $q$ determines the edge,
and so does the fact that the two arcs coincide. Therefore, there are at
most $\lceil\alpha k_0\rceil$ edges in $C_k$ incident to $q$. This means that
the total weight is at most $\alpha^3k_0/k+\alpha^2/k$. Summing over all
values of $k$  and always using the better of the two estimates proved, we
obtain that the total weight of all edges in the sets $C_k$ incident to $q$ is
less than $\alpha^2(\log\alpha+2)$.

The total weight of all edges in $G$ incident to a fixed vertex $q\in X$ is
therefore at most $4\log t+\alpha^2\log\alpha+6\alpha^2$. The total weight of
all edges in $G$ is at most $|X|(4\log t+\alpha^2\log\alpha+6\alpha^2)$.

\subsection{Lower bound on total weight}

Now we fix a touching point $p\in T$. Let $a$ and $b$ be the two curves in $S$
that touch at $p$ and assume that both $a$ and
$b$ contain at least $k$ touching points to the right of $p$. This assumption
holds for all but at most $2nk$ touching points.

Fix $1\leq k\leq nt/2$. Let the next $k$ touching points along $a$ to the right of $p$ be
$r_1,\dots,r_k$. Let $b_0$ be the segment of $b$ to the right of $p$ and to
the left of
$r_k$. We assume that there are fewer than $k$ touching points on $b_0$. We
can assume this without loss of generality, as if this condition is violated,
we can switch the roles of $a$ and $b$.

For $i=1,\dots,k$, let $c_i\in S$ be the curve different from $a$ containing
$r_i$. Observe that $c_i$ must cross $b_0$ to the left of $r_i$ at least
once; see Figure \ref{Fig:LowerBoundSetup}. We are using here that all touchings are on the same side of any
curve. Let $q_i$ be any one of these crossing points.

\begin{figure}[htbp]
\begin{center}
\input{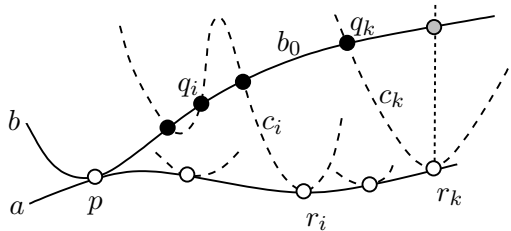}
\caption{\small Obtaining a lower bound on the total weight received by a touching $p$ between $a$ and $b$.
We consider the first $k$ touching points $r_i$, for $1\leq i\leq k$, along $a$ and to the right of $p$.  The other curve $c_i$ incident to $r_i$ must cross $b$ to the right of $p$ and to the left of $r_i$.}
\label{Fig:LowerBoundSetup}
\end{center}
\end{figure}

We claim that the total weight of the edges in $A_k$, $B_k$ and $C_k$
incident to $p$ is at least $\alpha$. We prove this claim by case
analysis.

{\bf Case 1.} $b_0$ contains at least $\alpha k$ points of $X$. Note that
all these points are connected to $p$ by an edge in $A_k$. These edges have a
total weight of at least $\alpha$, as claimed.

{\bf Case 2.} $q_i$ is connected to $p$ by an edge in $B_k$ for
$i=1,\dots,k$. Observe that these edges alone provide the total weight
$\alpha$, as claimed.

{\bf Case 3.} Neither Case 1, nor Case 2 applies. We fix an index $1\le
i\le k$ such that the crossing point $q_i$ is not connected to $p$ by an
edge in $B_k$. Note that condition B1 for this edge to be present in $B_k$ is
satisfied, thus condition B2 must be violated, i.e., there are at least
$k/\alpha$ distinct
touching points $p'$ on $b$ between $p$ and $q_i$ with the property that the
curve $d\in S\setminus\{b\}$ containing $p'$ touches $c_i$. Consider the
triangle-like region bounded by the segment of $a^*$ of $a$ between $p$ and
$r_i$, the segment $b^*$ of $b$ between $p$ and $q_i$ and the segment $c^*$ of
$c_i$ between $r_i$ and $q_i$. Close to $p'$, the curve $d$ is inside this
region, as depicted in Figure \ref{Fig:LowerBoundLense}. (We are using again
that the curves are touched from one side only.) As
$d$ touches both $b$ and $c_i$, it must leave the region through $a^*$. Let $q$
be the first point where  $d$ and $a^*$ cross to the left of $p'$, and let
$q'$ be the first point where they cross to the right of $p'$. Note that $q$ is
connected to $p$ by an edge in $C_k$. Indeed, conditions C1 and C2 are easy
to verify. Condition C3, holds as the segment of $a$ between $p$ and $q$ is
contained in $a^*$, which has exactly $i-1$ touching points in its interior.
Condition C4 holds, as the segment of $b$ to the right of $p$ and left
of $q'$ is
contained in $b_0$, which is assumed to contain fewer than $\alpha k$
points of $X$. Note that there were at least $k/\alpha$ distinct choices
for the touching point $p'$, each yielding a distinct edge in $C_k$ incident
to $p$, so the claim is also proved in this final case.

\begin{figure}[htbp]
\begin{center}
\input{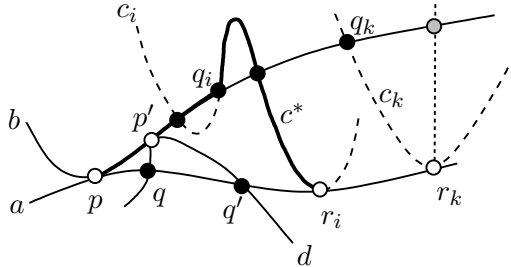}
\caption{\small Case 3. We encounter at least $k/\alpha$ curves $d$, each touching $b$ at a point $p'$ between $p$ and $q_i$.
For each $d$ of this kind, we set $q$ (resp., $q'$) to be the first crossing between $a$ and $d$ to the left (resp., right) of $p'$, and argue that the edge $(p,q)$ belongs to $C_k$.}
\label{Fig:LowerBoundLense}
\end{center}
\end{figure}

Summing for all the touching points in $T$, except the at most $2nk$
that are among the last $k$ on some curve, we get that the total weight of
the edges in $A_k$, $B_k$ and $C_k$ ($k$ fixed) is at least
$\alpha(|T|-2nk)$. Using that $|T|\ge tn/2$ and summing over all $k$, we
obtain the that the total weight of all edges in $G$ is at least $(\log
t-3)\alpha tn/2$. Comparing this with the upper bound obtained earlier, we get
$$|X|\ge{(\log t-3)\alpha tn/2\over4\log t+\alpha^2\log\alpha+6\alpha^2}.$$
Substituting $\alpha=\sqrt{\log t/\log\log t}$ finishes the proof of
Theorem~\ref{bibi}.

\section{Proof of Theorem~\ref{bipartite}}\label{fifth}

This proof is also based on the charging method, but we have to modify the
charging scheme used in the proof of Theorem~\ref{bibi}. We build the same
type of weighted bipartite graph between the vertex sets $X$ and $T$, where
$X=X_1\cup X_2$ and $X_i$ is the set of intersection points between curves in
$S_i$, while $T$ is the set of points where a curve of $S_1$ touches a curve
of $S_2$.

Note that each curve in $S=S_1\cup S_2$ is touched from one side only. Indeed,
a curve touching a closed curve from one side cannot intersect a curve
touching the same closed curve from the other side. We orient each curve $a\in
S$ so that the curves touching it lie on its left. This
orientation specifies which of the two arcs of a closed curve $a$
connecting two distinct points $p,q\in a$ is considered {\em the arc of $a$
from $p$ to $q$}. We consider this arc relatively open, thus $p$ and $q$ do
not belong to the arc connecting them.

We define the sets of weighted edges, $A_k$, $A'_k$, $A''_k$, $B_k$ and
$C_k$, running between $T$ and $X$, whose definition is detailed below and illustrated in Figures \ref{Fig:ChargeAGeneral}, \ref{Fig:ChargeBGeneral} and \ref{Fig:ChargeCGeneral}. The edge set of $G$ is defined as the
disjoint union of the sets $A_k$, $A'_k$, $A''_k$, $B_k$ and $C_k$, for all $1\leq k<n$ that
are powers of $2$. Thus we have $l:=\lceil\log n\rceil$ distinct values of $k$
to consider. The parameter $\alpha>1$ plays the same role as in the proof of
Theorem~\ref{bibi} and will be set later.

\begin{figure}[htbp]
\begin{center}
\input{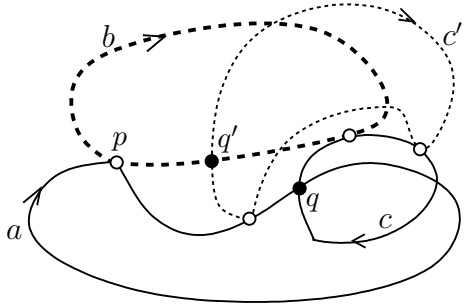}
\caption{\small The construction of $A_k, A'_k$ and $A''_k$. A touching $p\in T$ between $a\in S_1$ and $b\in S_2$ is connected in $A_k$ to a crossing $q\in X_1$ between $a$ and a third curve $c\in S_1$ if and only if the arc of $a$ from $p$ to $q$ contains fewer than $k$ points of $T$. Each edge $(p,q')\in A'_k\cup A''_k$ connects a touching $p\in T$ between $a\in S_1$ and $b\in S_2$ to a crossing $q'\in X_2$ between $b$ and some third curve $c'\in S_2$.}
\label{Fig:ChargeAGeneral}
\end{center}
\end{figure}

\medskip
\noindent {\bf A.} The edge connecting a touching $p\in T$ to a crossing $q\in
X_1$ is present in $A_k$ with weight $1/k$ if $p$ and $q$ lie on a common
curve $a\in S_1$ and the arc of $a$ from $p$ to $q$ contains fewer than $k$
touching points.

\medskip
\noindent {\bf A'.} The edge connecting $p\in T$ and $q\in X_2$ is present in
$A'_k$ with weight $\alpha/k$ if $p$ and $q$ lie on a common curve $b\in
S_2$ and the arc of $b$ from $q$ to $p$ contains fewer than $k/\alpha$
touching points.

\medskip
\noindent {\bf A''.} The edge connecting $p\in T$ and $q\in X_2$ is present in
$A''_k$ with weight $1/(\alpha k)$ if $p$ and $q$ lie on a common curve $b\in
S_2$ and the arc of $b$ from $q$ to $p$ contains fewer than $\alpha k$
touching points.

Note that if $\alpha^2$ is a
power of $2$, then $A''_k$ and $A'_{\alpha^2k}$ are identical.

\begin{figure}[htbp]
\begin{center}
\input{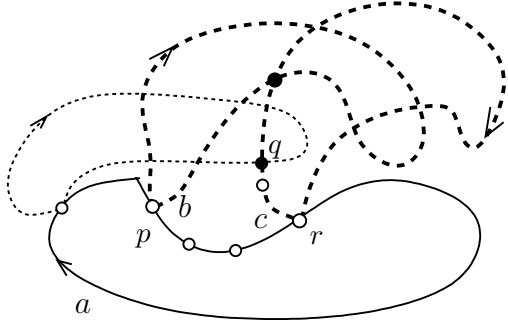}
\caption{\small The construction of $B_k$. A touching $p\in T$ between $a\in S_1$ and $b\in S_2$ is connected to a crossing $q\in X_2$ if
there is a touching $r$ between $a$ and some $c\in S_2$ so that the arc of $a$ from $p$ to $r$ contains fewer than $k$ points of $T$ and the arc of $c$ from $r$ to $q$ contains fewer than $\alpha k$ points of $T$.}
\label{Fig:ChargeBGeneral}
\end{center}
\end{figure}

\medskip
\noindent {\bf B.} The edge connecting $p\in T$ and $q\in X_2$ is present in
$B_k$ with weight $1/(\alpha k^2)$ if there is a touching $r\in T$ between
some curves $a\in S_1$ and $c\in S_2$ satisfying that

{\bf B1.} $p$ is on $a$ and the arc of $a$ from $p$ to $r$ contains fewer than
$k$ touching points and

{\bf B2.} $q$ is on $c$ and the arc of $c$ from $r$ to $q$ contains fewer than
$\alpha k$ touching points.

\medskip
\noindent {\bf C.} Let $q$ be an intersection point of two curves $a$ and
$d$ belonging to $S_1$. Let $q'$ be the intersection of $a$ and $d$ coming
right after $q$ along $d$. That is, the arc of $d$ from $q$ to $q'$ is
disjoint from $a$ (see Figure \ref{Fig:ChargeCGeneral}). The edge connecting
$p$ and $q$ is present in $C_k$ with weight $\alpha^2/k$ if

{\bf C1.} $p$ lies on $a$ and the arc of $a$ from
$p$ to $q$ contains fewer than $k$ points of $T$,

{\bf C2.} the arc of $d$ from $q$ to $q'$ contains fewer than $3\alpha^2k$ points
of $T$ and

{\bf C3.} the curve $b\in S_2$ through $p$ touches $d$ within the arc of $d$
from $q$ to $q'$.

\begin{figure}[htbp]
\begin{center}
\input{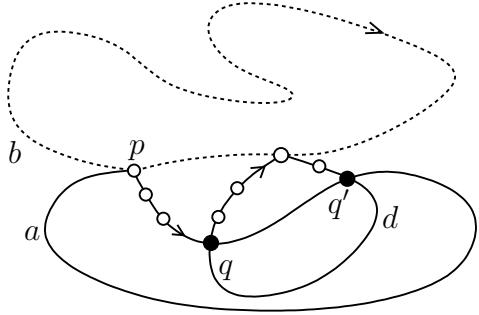}
\caption{\small The construction of $C_k$. A touching $p\in T$ between $a\in S_1$ and $b\in S_2$ is connected to a crossing $q\in X_1$ between $a$ and some other curve $d\in S_1$ if the arc of $a$ from $p$ to $q$ contains fewer than $k$ points of $T$, the arc of $d$ between $q$ and the next crossing $q'$ of $a$ and $d$ contains less than $3\alpha^2k$ points of $T$, and $b$ touches $d$ within its arc from $q$ to $q'$.}
\label{Fig:ChargeCGeneral}
\end{center}
\end{figure}

\medskip
In the case $p$ and $q$ are
connected in several of these sets, we consider the corresponding edges
distinct and $G$ contains several parallel edges connecting $p$ and $q$,
possibly with different weights.

\subsection{Upper bound on total weight}

We start by considering the edges incident to a fixed crossing point
$q\in X$.

A vertex $q\in X_1$ is incident to at most $2k$ edges in $A_k$ for a total weight of at most $2$. Summing
over all possible values of $k$, we get a total weight of at most $2l$.

A vertex $q\in X_2$ is incident to at most $2\lceil k/\alpha\rceil$ edges in
$A'_k$ and at most $2\lceil\alpha k\rceil$ edges of $A''_k$ for a total weight
of less than $4+4\alpha/k$. Summing over all $k$, we obtain that
the total weight of these edges is less than $4l+8\alpha$.

In a similar way, given a vertex $q\in X_2$, we can choose $r\in T$ to satisfy
condition B2 in $2\lceil\alpha k\rceil$ ways and each such vertex $r$ brings
about $k$ edges of $B_k$ incident to $q$ for a total of at most $2k\lceil
\alpha k\rceil$ edges with a total wight less than $2+2/(\alpha k)$. Summing
over all $k$, we obtain that the total weight of these edges is less than
$2l+4/\alpha$.

Finally, we consider the edges in $C_k$. Let us fix $q\in X_1$. This point
lies on two curves in $S_1$. To produce an edge in $C_k$ incident to $q$, one
must play the role of $a$ and the other the role of $d$. By condition C3, a point
of the arc of $d$ from $q$ to $q'$ is shared by a curve in $S_2$, so it must
be on the left side of $a$. This also holds for the entire arc, as it is disjoint
from $a$. This property uniquely determines which of the two curves
through $q$ must play the role of $a$ and which one must play the role of $d$. Thus, $q'$ is
also determined by $q$ and is independent of $k$ and of the particular edge in
$C_k$ incident to $q$.

Condition C1 guarantees that the number of edges in $C_k$ incident to $q$ is at
most $k$, for a total weight of at most $\alpha^2$.

Let $k_0$ be the smallest index with an edge of $C_{k_0}$ incident to a given
vertex $q\in X_1$. By condition C2, the arc of $d$ from $q$ to $q'$ contains at
most $3\alpha^2k_0$ points of $T$. If $p$ is adjacent to $q$ in $C_k$ for some
$k$, then the curve $b\in S_2$ through $p$ must touch $d$ in its arc from $q$
to $q'$ and this touching point determines $p$. Thus, the number of edges in
$C_k$ ($k$ fixed) incident to $q$ is at most $3\alpha^2k_0$ for a total weight of $3\alpha^4k_0/k$.

Summing over all values of $k\ge k_0$  and always using
the better of the two estimates proved we get the total weight of all edges
in the sets $C_k$ incident to $q$ is less than $2\alpha^2\log\alpha+4\alpha^2$.

The total weight of all edges in $G$ incident to a fixed vertex $q\in X$ is
therefore at most $6l+2\alpha^2\log\alpha+12\alpha^2$. The
total weight of all edges in $G$ is at most $|X|$ times this value.

\subsection{Lower bound on total weight}

Fix a touching point $p\in T$. Let $a\in S_1$ and $b\in S_2$ be the two
curves in $S$ that touch at $p$.

Let us fix $k$, and let the $k$ touching points that follow $p$ along $a$ (in the chosen
orientation) be denoted by $r_1,\dots,r_k$ (with $r_1$ being closest). Note
that $a$ contains exactly $n$ points in $T$, so $p$ and $r_1,\dots,r_k$ are
all distinct.

\begin{figure}[htbp]
\begin{center}
\input{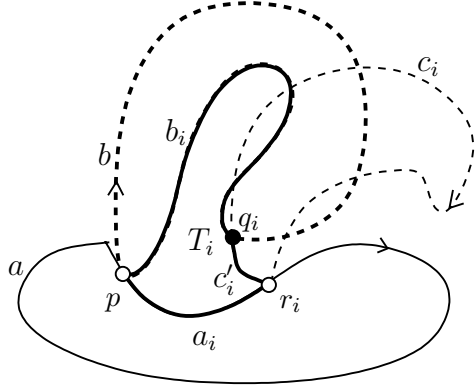}
\caption{\small Lower bound for the proof of Theorem \ref{bipartite}: the overall setup. $q_i$ is the first intersection point of $c_i$ and $b$ after $r_i$ along $c_i$. $T_i$ is the region whose boundary is composed of the arc $b_i$ of $b$ from $q_i$ to $p$, the arc $c'_i$ of $c_i$ from $r_i$ to $q_i$, and the arc $a_i$ of $a$ from $p$ to $r_i$.}
\label{Fig:LowerBoundGeneral}
\end{center}
\end{figure}

For every $i=1,\dots,k$, let $c_i\in S_2$ be the curve containing
$r_i$. Let $q_i$ be the first intersection of the curves $b$ and $c_i$ after
$r_i$, with respect to the orientation of $c_i$. Consider the closed curve
consisting of the following three arcs: the arc $b_i$ of $b$ from $q_i$ to
$p$, the arc $c'_i$ of $c_i$ from $r_i$ to $q_i$, and the arc $a_i$ of $a$
from $p$ to $r_i$. This (together with the points $p$, $q_i$ and $r_i$) is a
simple closed curve and, hence, partitions the plane into two connected
components; see Figure \ref{Fig:LowerBoundGeneral}. We denote by $T_i$ the component that lies on the left side of the
three arcs. Note that the curves $a$ and $b$ are disjoint from $T_i$, but $c_i$
may enter $T_i$ through the arc $b_i$ several times.

We claim that the total weight of the edges in $A_k$, $A'_k$, $A''_k$, $B_k$
and $C_k$ incident to $p$ is at least $\alpha$. We prove this claim by case
analysis.

\smallskip

{\bf Case 1.} The arc $a_k$ contains at least $\alpha k$
points of $X$. Note that all these points are connected to $p$ by an edge in
$A_k$. These edges have a total weight of at least $\alpha$, as claimed.

Note that any curve in $S_1$ touching $b$ in $b_i$ or touching $c_i$ in $c'_i$
must be inside $T_i$ in a small neighborhood of this touching point. But as it
intersects $a$ it cannot be entirely within $T_i$, so it must cross the
boundary of $T_i$. In particular it must cross $a_i$, as it can only touch $b$
and $c_i$ (due to the complete bipartite touching structure of $S_1\times S_2$). See Figure \ref{Fig:LowerBoundGeneralCases} (left). The curve $a_i$ is contained in $a_k$. Thus, if Case 1 does not hold,
$b_i$ and $c_i'$ contain fewer than $\alpha k$ points from $T$, for any $1\le i\le
k$.

\begin{figure}[htbp]
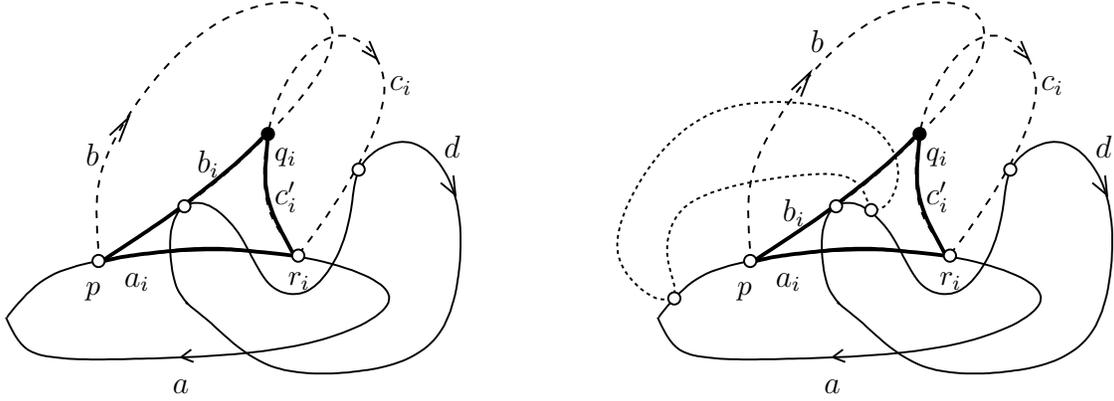

\begin{center}
\input{LowerBoundGeneral1.pstex_t}\hspace{2cm}\input{LowerBoundGeneral5.pstex_t}
\caption{\small Lower bound for the proof of Theorem \ref{bipartite}. Left: Any curve in $S_1$ that touches $b$ or $c_i$ within the respective arcs $b_i$ or $c'_i$ must cross $a$ within $a_i$. Therefore, if Case 1 does not occur, the overall number of such curves cannot exceed $\alpha k$. Right: In Case 4, the curve $b$ is touched within $b_i$ by at least $k/\alpha$ curves $d$ of $S_1$. Notice that any curve of $S_2$ that touches $d$ within $T_i$ must enter that region through $b_i$ or $c'_i$. Since none of the Cases 1--3 holds, the overall number of such curves cannot exceed $3\alpha^2 k$.}
\label{Fig:LowerBoundGeneralCases}
\end{center}
\end{figure}

\smallskip
{\bf Case 2.} Case 1 does not hold and $b_i$ contains at least $\alpha^2k$ points
from $X$, for some $i$. As $b_i$ has fewer than $\alpha k$ points of $T$, every
point of $b_i$ that belongs to $X$ is connected to $p$ by an edge in $A''_k$, providing
a total weight $\alpha$, as required.

\smallskip
{\bf Case 3.} Case 1 does not hold and for every $i=1,\ldots,k$, either
$c'_i$ contains at least $2\alpha^2k$ points of $X$ or $b_i$ contains fewer
than $k/\alpha$ points of $T$. Note that $a_i$ contains $i-1<k$ points of $T$,
while $c'_i$ contains fewer than $\alpha k$ points of $T$. Hence, every
crossing point on $c'_i$ is connected to $p$ by an edge in $B_k$. This
represents a total weight of at least $2\alpha/k$ if $|c'_i\cap
X|\ge2\alpha^2k$. For values of $i$ with $|b_i\cap T|<k/\alpha$, the vertex
$q_i$ is adjacent to $p$ in $A'_k$, and this edge alone has weight
$\alpha/k$. If either possibility happens for all $k$ possible values of $i$,
these weights add up, except that some edges in $B_k$ may be counted
twice. This results in a total weight of at least $\alpha$ for the edges in
$A'_k$ and $B_k$ incident to $p$.

\smallskip
{\bf Case 4.} None of the Cases 1, 2 or 3 hold. We fix an index $1\le
i\le k$ such that $c'_i$ contains fewer than $2\alpha^2k$ points of $X$ and $b_i$
contains at least $k/\alpha$ points of $T$. This is possible, as Case 3 does
not hold. Let us consider one of the curves
in $S_1$ touching $b$ at a point $z$ of $b_i$. As was pointed out earlier, this curve
$d$ lies inside $T_i$ in a small neighborhood of $z$ and leaves $T_i$ by
crossing $a_i$. Let $q$ be the last crossing between $d$ and $a$ before $z$,
and let $q'$  be the first crossing between them after $z$, with respect to the
orientation of $d$. See Figure \ref{Fig:LowerBoundGeneralCases} (right) for an illustration.

We claim that $q$ is adjacent to $p$ in $C_k$. Indeed, conditions C1 and C3
hold trivially. To verify C2, consider a curve $e\in S_2$ touching $d$ at a point
in the arc of $d$ from $q$ to $q'$. This touching point (as the whole arc of
$d$ from
$q$ to $q'$) lies inside $T_i$. However, $e$ must intersect $b$, so it must leave
$T_i$. It touches $a$, so it must leave through either $b_i$ or $c'_i$. Here
$c'_i$ has fewer than $2\alpha^2k$ points in $X$, by our choice of $i$,
while $b_i$ contains fewer than $\alpha^2k$ points in $X$, since Case 2
does not hold. Hence, $e$ has fewer than $3\alpha^2k$ points where it can leave
$T_i$, and there are at most $3\alpha^2k$ possible choices for the curve
$e$. This means that condition C2 is satisfied and $q$ is adjacent to $p$ in
$C_k$.

By our choice of $i$, we can select the curve $d$ in at least $k/\alpha$ different ways, each giving rise to a different edge in $C_k$ incident to $p$. The
total weight of these edges is $\alpha$. This completes the analysis of the last case, showing that the total weight of all edges in $A_k$, $A'_k$,
$A''_k$, $B_k$ and $C_k$ incident to $p$ is at least $\alpha$.
\smallskip

Summing over all $l$ possible values of $k$ and over all
$n^2$ touching points in $T$, we conclude that the total weight of $G$ is at least
$\alpha\lceil\log n\rceil n^2$. Comparing this lower bound with the upper
bound proved in the preceding subsection and substituting $\alpha=\sqrt{\log
n/\log\log n}$, Theorem~\ref{bipartite} follows.

\end{document}